\documentclass[11pt,reqno]{amsart} 
\usepackage{asymptote}
\usepackage[labelfont=bf]{caption}
\usepackage[utf8]{inputenc}
\usepackage{amsfonts,amsmath,amsthm,amssymb,fullpage}
\usepackage[justification=centering]{caption}
\usepackage{color}
\usepackage{enumitem}  
\usepackage{hyperref}
\usepackage{mathtools}
\usepackage{tikz}
\usepackage{xcolor}
\usepackage{thm-restate}
\usepackage{thmtools}
\usepackage[backend=biber,sorting=nty,style=numeric,maxbibnames=99]{biblatex}
\makeatletter
\newcommand{\addresseshere}{%
  \enddoc@text\let\enddoc@text\relax
}
\makeatother

\newtheorem{theorem}{Theorem}[section]

\newtheorem{lemma}[theorem]{Lemma}

\newtheorem{claim}[theorem]{Claim}
\newtheorem{conjecture}[theorem]{Conjecture}
\theoremstyle{definition}

\newtheorem{definition}[theorem]{Definition}
\title{An Improved Bound for Smith's Longest Cycles Conjecture via a Forbidden Subdivision 
}

\author[Chen]{Douglas M. Chen}
\address[D.~M.~Chen]{Columbia University, New York, NY 10027, USA}
\email{\textcolor{blue}{\href{mailto:dmc2265@columbia.edu}{dmc2265@columbia.edu}}}
\begin{document}

\begin{abstract}
Smith's conjecture asserts that in every $k$-connected graph with $k\geq 2$, any two longest cycles intersect in at least $k$ vertices.
In this work, we establish an $\Omega(k^{8/11})$ bound for this conjecture, improving upon the $\Omega(k^{2/3})$ bound of Ma and Zhao.
Our proof combines a Ramsey theoretic refinement of the traditional Tur\'an-type approach with computer search.
\end{abstract}

\maketitle
\section{Introduction}
The study of longest cycles is a classical part of extremal graph theory, yet several of its most fundamental questions remain open.
One such question from 1984, attributed to Smith, concerns the relationship between the connectivity of a graph and intersections of longest cycles \cite{bondy1995graphs}.
\begin{conjecture}[Smith]
\label{conj: smith}
    In every $k$-connected graph with $k\geq 2$, any two longest cycles intersect in at least $k$ vertices.
\end{conjecture}
\subsection{Previous Results and Approaches}
Smith's conjecture is known for several special cases, including when $k\leq 6$ by Gr\"otschel \cite{grotschel1984smith}, when $k \in \{ 7,8 \}$ by Stewart and Thompson \cite{stewart1995smith}, and when $k\geq \frac{n+16}{7}$, where $n$ is the order of the graph, by Guti\'errez and Valqui \cite{gutierrez2024smith}.

Concerning the general case, there has been a series of improvements via constant exponent steps towards the conjecture.
In 1995, Burr and Zamfirescu established an $\Omega(k^{1/2})$ bound \cite{bondy1995graphs} by establishing edge bounds for a natural auxiliary graph associated with the conjecture.
This was soon improved in 1998 by Chen, Faudree, and Gould, who showed an $\Omega(k^{3/5})$ bound \cite{chen1998smith}.
They take a Tur\'an-type approach by excluding a fixed subgraph, namely the complete bipartite graph $K_{3,257}$, from the auxiliary graph.
The idea is then to demonstrate that the presence of such a subgraph supports a rerouting that combines portions of the two longest cycles with additional edges to reveal a longer cycle, yielding a contradiction.
This rerouting is constructed by hand.

Nearly three decades later, in 2025, Groenland, Longbrake, Steiner, Turcotte, and Yepremyan established an $\Omega(k^{5/8})$ bound \cite{groenland2025smith}.
They also utilize a Tur\'an-type approach, excluding the $3$-cube with one added diagonal and $K_{3,3}$ from the auxiliary graph.
Additionally, they developed a new computational approach to construct the rerouting by encoding endpoint orders, reducing via symmetry, and resolving the remaining cases via computer search and linear programming.

A few months later, Ma and Zhao showed an $\Omega(k^{2/3})$ bound \cite{ma2025smith}.
They take a different approach by combining supersaturation estimates with an upper bound on the number of copies of $K_{2,7}$ in the auxiliary graph.
This upper bound is proven using an elegant combinatorial rerouting method.

For additional background on Smith's conjecture and related topics, see the surveys \cite{shabbir2013survey,zamfirescu2001survey}.
\subsection{Our Contribution}
We prove the following improved bound for Smith's conjecture, where we push the exponent of $k$ from $\frac23=0.6666\dots$ to $\frac{8}{11}=0.7272\dots$.
\begin{theorem}
\label{thm: main}
    In every $k$-connected graph with $k\geq 2$, any two longest cycles intersect in $\Omega(k^{8/11})$ vertices.
\end{theorem}
The ideas from \cite{chen1998smith} and the computer search method of \cite{groenland2025smith} are conceptually the most relevant to our approach.
The main limitation of the Tur\'an-type approach is that the auxiliary graph records only where two longest cycles meet and not the order of the corresponding endpoints along the cycles.
However, this order determines whether a proposed rerouting actually produces a longer cycle needed for the contradiction.
A construction that is valid for one endpoint ordering may get ``tangled" for another ordering and split into several smaller disjoint cycles.
Thus, excluding a new subgraph requires machinery to deal with an enormous family of possible endpoint orderings.

We exclude a larger subgraph, namely the subdivision of a $K_{4,T}$, by again proceeding by contradiction.
Assuming that this subgraph exists, our key new idea is that we can choose $T$ large enough to extract a small subgraph with very well-behaved endpoint orderings using Ramsey theory.
We can then guide our rerouting construction using this small subgraph, which allows us to reduce the problem down to a reasonable number of cases that can be resolved with computer search, thereby establishing the desired contradiction.
\subsection{Roadmap}
The rest of the paper is organized as follows.
We first review several relevant facts in Section \ref{section: prelim}.
In Section \ref{section: proof}, we formalize our computer search method and prove Theorem \ref{thm: main} conditional on a routing claim.
In Appendix \ref{section: computation}, we provide implementation details for our computer search and verify the missing claim.
\section{Preliminaries and Conventions}
\label{section: prelim}
Throughout the rest of this work, we let $G$ be a $k$-connected graph with $k\geq 2$ and let $X$ and $Y$ be two longest cycles of $G$ with $|X|=|Y|=\ell$.
Moreover, let $M=V(X) \cap V(Y)$ and $m=|M|$.
Note that if $m\geq k$, then our main result follows immediately, so we assume that $m<k$.

\subsection{Fragments and Joining Paths}
We now recall the construction of the auxiliary graph studied in previous works (cf. Definition 2.3 in \cite{ma2025smith}, \S 3 in \cite{groenland2025smith}, p. 146 in \cite{chen1998smith})
To our knowledge, this construction has no standard name, so for clarity, we introduce the terminology below.

First, we define the vertices.
Call each connected component of the graph $X-M$ an \emph{$X$-fragment} and each connected component of the graph $Y-M$ a \emph{$Y$-fragment}.
Let $\mathcal X$ denote the set of $X$-fragments and $\mathcal Y$ denote the set of $Y$-fragments.
Note that the graphs $X-M$ and $Y-M$ are disjoint unions of paths, so the elements of $\mathcal X$ and $\mathcal Y$ are paths in $G$.

Next, we define the edges.
Choose a maximum cardinality family $\mathcal J$ of pairwise vertex-disjoint paths in the graph $G-M$ such that each has one endpoint in $X-M$ and the other endpoint in $Y-M$, and shorten each path so that its interior avoids $X \cup Y$.
Call each resulting path a \emph{joining path}.

The \emph{fragment incidence graph}, denoted $\mathcal F=\mathcal F(\mathcal X,\mathcal Y,\mathcal J)$, is the bipartite graph with bipartition $\mathcal X\sqcup\mathcal Y$ and edge set $\mathcal J$.

The reason why $\mathcal F$ is a natural auxiliary graph to study is because of the following nice result, which guarantees that $\mathcal F$ is well-defined and is constrained by $k$ and $|V(X) \cap V(Y)|=m$.
\begin{lemma}
\label{lem: auxgraph}
    The bipartite graph $\mathcal F$ is simple, and $|V(\mathcal F)|\leq 2m$ and $|E(\mathcal F)|\geq k-m$.
\end{lemma}
One can prove this result using Menger's theorem and a cycle exchange argument (cf. Proposition 2.2 in \cite{ma2025smith}).
Moreover, one can readily show as a corollary that $m\geq \sqrt{k}-1$ and hence derive Burr and Zamfirescu's $\Omega(k^{1/2})$ bound for Smith's conjecture.

\subsection{Subdivisions}
Recall that the \emph{subdivision} of a graph $H$ is the graph obtained by replacing all edges of $H$ by internally disjoint paths of length $2$. 
Let $H'$ denote the subdivision of $H$.

We are interested in subdivisions of complete bipartite graphs.
Given the subdivision $H'$ of a complete bipartite graph $H$, the vertices of $H'$ inherited from $H$ are called \emph{branch vertices}, the two partite sets of $H$, considered as subsets of $V(H')$, are called \emph{branch classes}, and the vertices introduced by subdividing the edges of $H$ are called \emph{subdividing vertices} \cite{hou2026subdiv}.  

We will invoke the following general result of Conlon, Janzer, and Lee \cite{conlon2021subdiv} on subdivisions of complete bipartite graphs.
\begin{theorem}[Conlon, Janzer, and Lee]
\label{thm: exsubdiv}
    For any integers $2\leq s\leq t$, every graph on $n$ vertices that does not contain a $K_{s,t}'$ has $O(n^{\frac32-\frac{1}{2s}})$ edges.
\end{theorem}
This bound is tight when $t$ is sufficiently large compared to $s$.
\section{Proof of Theorem \ref{thm: main} and Reduction to Computer Search}
\label{section: proof}
In this section, we prove Theorem \ref{thm: main} conditional on a routing claim (cf. Claim \ref{claim: route}).
We also formalize our computer search approach that will prove this claim (cf. Appendix \ref{section: computation}).

\subsection{Setup} 
The proof of Theorem \ref{thm: main} follows readily once the following forbidden subgraph result is established.
\begin{theorem}
\label{thm: nosubdiv}
    There is an absolute integer $T$ such that the fragment incidence graph $\mathcal F=\mathcal F(\mathcal X, \mathcal Y, \mathcal J)$ contains no $K_{4,T}'$.
\end{theorem}
We will prove Theorem \ref{thm: nosubdiv} via contradiction.
In particular, we assume that $\mathcal{F}$ contains a $K_{4,T}'$, where $T$ is a sufficiently large number to be determined later, and the goal is to utilize this subgraph to expose a cycle that is longer than either of $X$ and $Y$.

To extract useful information from this subgraph, we introduce the following definitions.
\begin{definition}
    Let $C \in \{ X,Y \}$. 
    A \emph{rooted orientation} of $C$ is a pair $\theta_C=(c_0,\delta_C)$, where $c_0 \in M$ and $\delta_C$ is one of the two cyclic directions around $C$ (i.e. clockwise or counterclockwise).
    For vertices $u,v \in V(C) \setminus \{ c_0 \}$, write $u\prec_{\theta_C} v$ when $u$ is passed before $v$ while traversing $C$ starting from $c_0$ in direction $\delta_C$.
    For distinct $C$-fragments $W_1$ and $W_2$, write $W_1\prec_{\theta_C} W_2$ when the vertices of $W_1$ are passed before the vertices of $W_2$ while traversing $C$ starting from $c_0$ in direction $\delta_C$.
\end{definition}
Note that $\prec_{\theta_C}$ defines a strict total order on both $V(C) \setminus \{ c_0 \}$ and the set of $C$-fragments.
When $\theta_X$ and $\theta_Y$ are understood, we abbreviate $\prec_{\theta_X}$ as $\prec_X$ and $\prec_{\theta_Y}$ as $\prec_Y$.
\begin{definition}
\label{def: arrangement}
    A \emph{$K_{4,t}'$-arrangement} in $\mathcal F$ consists of a subgraph $\mathcal H\cong K_{4,t}'$ of $\mathcal F$ with branch classes $\mathcal B=\{ B_1,B_2,B_3,B_4 \}\subseteq \mathcal X$ and $\mathcal A=\{ A_1,A_2,\dots,A_t \} \subseteq \mathcal X$, and rooted orientations $\theta_X=(x_0,\delta_X)$ and $\theta_Y=(y_0,\delta_Y)$ of $X$ and $Y$, respectively, such that $x_0 \in M$ immediately precedes $B_1$ in direction $\delta_X$ and 
    \[ B_1\prec_X A_1\prec_X A_2\prec_X\cdots\prec_X A_t\prec_X B_2\prec_X B_3 \prec_X B_4. \]
    Let $Y_{ij}$ denote the subdividing vertex of the edge $B_iA_j$ in $\mathcal F$.
    Let $P_{ij}^{\mathcal B} \in \mathcal J$ and $P_{ij}^{\mathcal A}\in \mathcal J$ denote the joining paths corresponding to the edges $B_iY_{ij}$ and $A_jY_{ij}$ of $\mathcal F$, respectively.
    Denote the endpoints of $P_{ij}^{\mathcal B}$ as $b_{ij} \in B_i$ and $y_{ij}^+ \in Y_{ij}$ and the endpoints of $P_{ij}^{\mathcal A}$ as $a_{ij} \in A_j$ and $y_{ij}^- \in Y_{ij}$.
\end{definition}
\begin{figure}
    \centering
    \includegraphics[width=10cm]{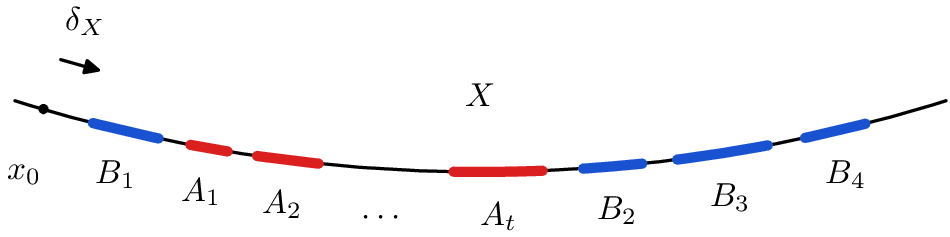}
    \caption{Illustration of a rooted orientation $\theta_X=(x_0,\delta_X)$ in a $K_{4,t}'$-arrangement.}
\end{figure}
\subsection{Untangling with Ramsey}
Our next step is to identify a subgraph with very well-behaved endpoint orders.
This will anchor our computation to a reasonable number of cases.
\newpage
\begin{definition}
\label{def: homogeneous}
    A $K_{4,t}'$-arrangement is \emph{homogeneous} if 
    \begin{enumerate}[label=(\roman*)]
        \item there is a single permutation $\pi \in \mathfrak S_4$ such that for all $j \in [t]$, we have
        \[ a_{\pi(1),j}\prec_X a_{\pi(2),j}\prec_X a_{\pi(3),j}\prec_X a_{\pi(4),j}, \]
        \label{condition (i)}
        \item for all $i \in [4]$, either $b_{i1}\prec_X b_{i2}\prec_X \cdots \prec_X b_{it}$ or $b_{it}\prec_X b_{i,t-1}\prec_X \cdots \prec_X b_{i1}$,
        \label{condition (ii)}
        \item there is a single permutation $\sigma \in \mathfrak S_8$ such that for all $1\leq p<q\leq t$, we have
        \[ F_{\sigma(1)}\prec_Y F_{\sigma(2)}\prec_Y F_{\sigma(3)}\prec_Y F_{\sigma(4)}\prec_Y F_{\sigma(5)}\prec_Y F_{\sigma(6)}\prec_Y F_{\sigma(7)}\prec_Y F_{\sigma(8)}, \]
        where $F_i=Y_{ip}$ and $F_{i+4}=Y_{iq}$ for all $i \in [4]$, and
        \label{condition (iii)}
        \item for all $i \in [4]$, either $y_{ij}^+\prec_Y y_{ij}^-$ for all $j \in [t]$ or $y_{ij}^-\prec_Y y_{ij}^+$ for all $j \in [t]$.
        \label{condition (iv)}
    \end{enumerate}
\end{definition}
\begin{lemma}
\label{lem: extract}
    There is an absolute integer $T$ such that if $\mathcal F$ contains a copy of $K_{4,T}'$ whose branch vertices are $X$-fragments, then $\mathcal F$ contains a homogeneous $K_{4,8}'$-arrangement.
\end{lemma}
\begin{proof}
    Set $r=(4!)^2 \cdot 8! \cdot 2^{12}$, and choose $T\geq 4R_r(8)$ (here, $R_r(8)$ is the $r$-color Ramsey number for a monochromatic $K_8$).
    Let $\mathcal B\subseteq \mathcal X$ and $\mathcal A_0\subseteq \mathcal X$ denote the two classes of branch vertices in the given copy $\mathcal H\cong K_{4,T}'$, where $|\mathcal B|=4$ and $|\mathcal A_0|=T$.

    Now, fix a direction $\delta_X$ around $X$.
    Note that the fragments of $\mathcal B$ divide $X$ into four cyclic intervals.
    Since each fragment of $\mathcal A_0$ lies in one of these intervals, the pigeonhole principle implies that there is an interval $\mathcal I$ containing at least $\frac{T}{4}$ fragments of $\mathcal A_0$.
    Let $B_1$ and $B_2$ be the fragments of $\mathcal B$ immediately preceding and following $\mathcal I$, respectively, in direction $\delta_X$, and fix $x_0 \in M$ immediately before $B_1$.
    This defines a rooted orientation of $X$, namely $\theta_X=(x_0,\delta_X)$. 
    Label the remaining fragments of $\mathcal B$ as $B_3$ and $B_4$, and list the fragments of $\mathcal A_0$ contained in $\mathcal I$ as $A_1,A_2,\dots,A_t$, so that
    \[ B_1\prec_X A_1\prec_X A_2\prec_X\cdots\prec_X A_t\prec_X B_2\prec_X B_3\prec_X B_4 , \]
    where $t\geq \frac{T}{4}\geq R_r(8)$.
    Choose any rooted orientation $\theta_Y$ of $Y$.
    Then we have a $K_{4,t}'$-arrangement with branch classes $\mathcal B=\{ B_1,B_2,B_3,B_4 \}\subseteq \mathcal X$ and $\mathcal A=\{ A_1,A_2,\dots,A_t \} \subseteq \mathcal X$ and rooted orientations $\theta_X,\theta_Y$.
    
    Now, for each pair $A_p\prec_X A_q$ with $1\leq p<q\leq t$, define its color 
    \[ \chi(A_p,A_q)=(\chi_1(A_p,A_q),\chi_2(A_p,A_q),\chi_3(A_p,A_q),\chi_4(A_p,A_q),\chi_5(A_p,A_q),\chi_6(A_p,A_q)) \]
    as follows: 
    $\chi_1(A_p,A_q) \in \mathfrak S_4$ is the unique permutation for which 
    \[ a_{\chi_1(1),p}\prec_X a_{\chi_1(2),p}\prec_X a_{\chi_1(3),p}\prec_X a_{\chi_1(4),p}, \]
    $\chi_2(A_p,A_q) \in \mathfrak S_4$ is the unique permutation for which 
    \[ a_{\chi_2(1),q}\prec_X a_{\chi_2(2),q}\prec_X a_{\chi_2(3),q}\prec_X a_{\chi_2(4),q}, \]
    $\chi_3(A_p,A_q) \in \{ \pm 1 \}^4$ is the unique signing for which
    \[ \chi_3(i)=\begin{cases}
        +1 & \text{if $b_{ip}\prec_X b_{iq}$} \\
        -1 & \text{if $b_{iq}\prec_X b_{ip}$},
    \end{cases}
    \]
    $\chi_4(A_p,A_q) \in \{ \pm 1 \}^4$ is the unique signing for which
    \[ \chi_4(i)=\begin{cases}
        +1 & \text{if $y_{ip}^+\prec_Y y_{ip}^-$} \\
        -1 & \text{if $y_{ip}^-\prec_Y y_{ip}^+$},
    \end{cases} \]
    $\chi_5(A_p,A_q) \in \{ \pm 1 \}^4$ is the unique signing for which
    \[ \chi_5(i)=\begin{cases}
    +1 & \text{if $y_{iq}^+\prec_Y y_{iq}^-$} \\ 
    -1 & \text{if $y_{iq}^-\prec_Y y_{iq}^+$},
    \end{cases} \]
    and $\chi_6(A_p,A_q) \in \mathfrak S_8$ is the unique permutation for which 
    \[ F_{\chi_6(1)}\prec_Y F_{\chi_6(2)}\prec_Y F_{\chi_6(3)}\prec_Y F_{\chi_6(4)}\prec_Y F_{\chi_6(5)}\prec_Y F_{\chi_6(6)}\prec_Y F_{\chi_6(7)}\prec_Y F_{\chi_6(8)}, \]
    where $F_i=Y_{ip}$ and $F_{i+4}=Y_{iq}$ for all $i \in [4]$.

    Consider any such pair coloring.
    Note that there are $(4!)^2 \cdot (2^4)^3 \cdot 8!=r$ possible colors, and since $t\geq R_r(8)$, the multicolor Ramsey theorem implies that after a possible relabeling, the fragments $A_1\prec_X A_2\prec_X \cdots \prec_X A_8$ satisfy the following condition: for all $1\leq p<q\leq 8$, the colors $\chi(A_p,A_q)$ are the same.
    Thus, for all $1\leq p<q\leq 8$, we have
    \[ \chi(A_p,A_q)=(\chi_1^*,\chi_2^*,\chi_3^*,\chi_4^*,\chi_5^*,\chi_6^*) \]
    for some fixed $\chi_1^* \in \mathfrak S_4$, $\chi_2^* \in \mathfrak S_4$, $\chi_3^* \in \{ \pm 1 \}^4$, $\chi_4^* \in \{ \pm 1 \}^4$, $\chi_5^* \in \{ \pm 1 \}^4$, and $\chi_6^* \in \mathfrak S_8$.

    Now, consider $\hat{\mathcal H}\cong K_{4,8}'$, the subgraph of $\mathcal H\cong K_{4,T}'$ obtained by restricting to branch classes $\mathcal B=\{ B_1,B_2,B_3,B_4 \}$ and $\mathcal A=\{ A_1,A_2,\dots,A_8 \}$, together with $\theta_X,\theta_Y$.
    We claim that this comprises a homogeneous $K_{4,8}'$-arrangement.
    
    Indeed, unraveling definitions, it is immediate that $\sigma=\chi_6^*$ satisfies condition \ref{condition (iii)}.
    Next, repeatedly ordering each pair $b_{ip}$ and $b_{iq}$ with respect to $\prec_X$ using the common value of $\chi_3^*$ yields condition \ref{condition (ii)}.
    Moreover, $\chi_1(A_p,A_8)=\chi_1^*$ for all $p<8$ and $\chi_2(A_1,A_q)=\chi_2^*$ for all $q>1$, while by definition, we have the overlap $\chi_1(A_2,A_8)=\chi_2(A_1,A_2)$, so it follows that $\pi=\chi_1^*=\chi_2^* \in \mathfrak S_4$ satisfies condition \ref{condition (i)}.
    A similar overlapping argument shows that $\chi_4^*=\chi_5^*$, implying condition \ref{condition (iv)}.

    Therefore, with rooted orientations $\theta_X,\theta_Y$, we see that $\hat{\mathcal H}$ is a homogeneous $K_{4,8}'$-arrangement of $\mathcal F$. 
\end{proof}
We are now ready to prove the forbidden subgraph result.
\begin{proof}[Proof of Theorem \ref{thm: nosubdiv}]
    Assume not.
    After possibly exchanging the roles of $X$ and $Y$, Lemma \ref{lem: extract} implies that $\mathcal F$ contains a homogeneous $K_{4,8}'$-arrangement with rooted orientations $\theta_X=(x_0,\delta_X)$ and $\theta_Y=(y_0,\delta_Y)$ of $X$ and $Y$, respectively.
    Let $\mathcal H\cong K_{4,8}'$ be this copy with branch classes $\mathcal B=\{ B_1,B_2,B_3,B_4 \}\subseteq \mathcal X$ and $\mathcal A=\{ A_1,A_2,\dots,A_8 \}\subseteq \mathcal X$.

    Call a nonempty set $I \subseteq [4] \times [8]$ an \emph{even routing} if every $i \in [4]$ and every $j \in [8]$ occurs as a coordinate of an even number of elements of $I$.

    Let $I$ be an even routing with $|I|=s$ and $P(I)=\{ P_{ij}^{\mathcal B}: (i,j) \in I \} \cup \{ P_{ij}^{\mathcal A}:(i,j) \in I \}$.
    By construction, this is a family of $2s$ pairwise vertex-disjoint joining paths, each with one endpoint on $X$ and one endpoint on $Y$.
    For vertices $u,v \in V(X)$, let $X[u,v]$ denote the subpath of $X$ obtained by starting at $u$ and proceeding in direction $\delta_X$ until $v$ is reached; here, we allow the traversal to pass through $x_0$ if needed.
    Let $Y[u,v]$ denote the analogous subpath of $Y$.

    List the endpoints of $P(I)$ lying in $X$ as $x_1\prec_X x_2\prec_X \cdots\prec_X x_{2s}$.
    Since every branch fragment $B_i$ or $A_j$ is a contiguous subpath of $X-M$, the endpoints of $P(I)$ that lie in any fixed branch fragment form a consecutive block in the list.
    Moreover, the block corresponding to $B_i$ has size $|\{ j:(i,j) \in I \}|$, and the block corresponding to $A_j$ has size $|\{ i:(i,j) \in I \}|$, both of which are even since $I$ is an even routing.
    Thus, the last endpoint of each block must have even index, so the endpoints $x_{2h-1}$ and $x_{2h}$ lie in the same branch fragment for all $h \in [s]$.
    Similarly, list the endpoints of $P(I)$ lying in $Y$ as $y_1\prec_Y y_2\prec_Y \cdots\prec_Y y_{2s}$.
    For each $(i,j) \in I$, the subdividing fragment $Y_{ij}$ contains precisely the two endpoints $y_{ij}^+$ and $y_{ij}^-$, and since $Y_{ij}$ is a contiguous subpath of $Y-M$, these two endpoints form a consecutive block in the list.
    Thus, each block has size $2$, implying that the last endpoint of each block must have even index, so the $y_{2h-1}$ and $y_{2h}$ lie in the same subdividing fragment for all $h \in [s]$.

    Now, consider the two subgraphs
    \[ \Gamma_1(I)=\left( \bigcup_{P \in P(I)}P \right) \cup \left( \bigcup_{h=1}^{s}X[x_{2h-1},x_{2h}] \right) \cup \left( \bigcup_{h=1}^{s}Y[y_{2h},y_{2h+1}] \right) \]
    and 
    \[ \Gamma_2(I)=\left( \bigcup_{P \in P(I)}P \right) \cup \left( \bigcup_{h=1}^{s}X[x_{2h},x_{2h+1}] \right) \cup \left( \bigcup_{h=1}^{s}Y[y_{2h-1},y_{2h}] \right). \]
    Here, we consider each term as unions of subgraphs of $G$ and let $x_{2s+1}=x_1$ and $y_{2s+1}=y_1$.
    \begin{figure}
        \centering
        \includegraphics[width=10cm]{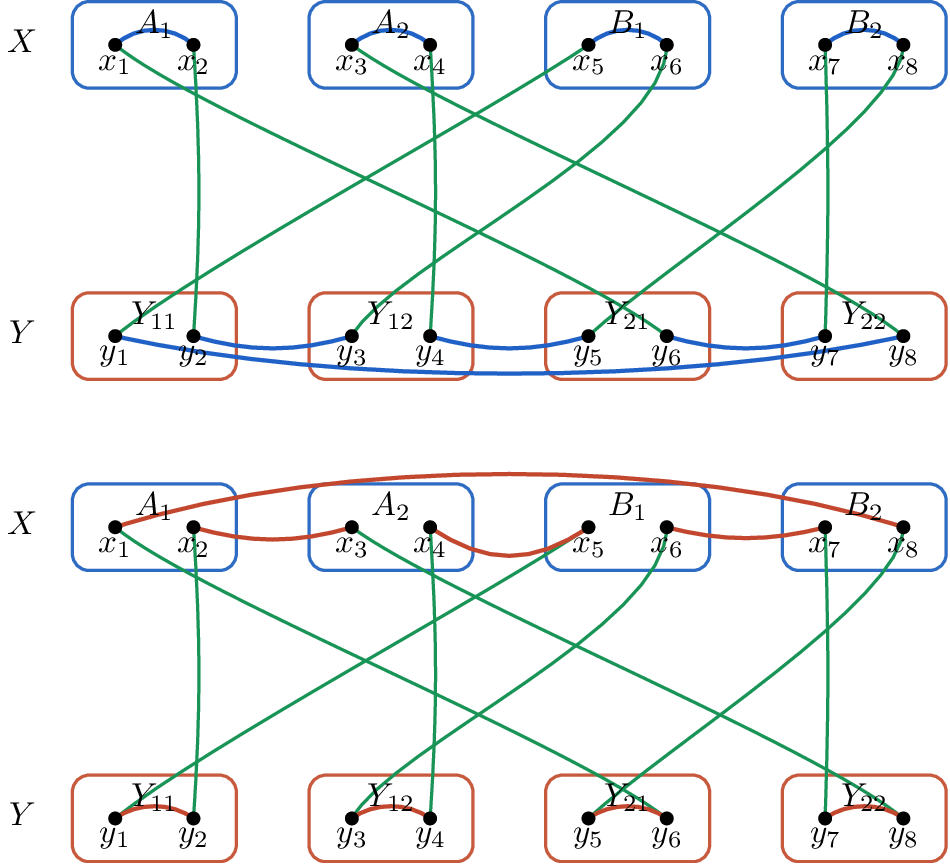}
        \caption{Schematic diagram of our subgraph constructions given an even routing $I\subseteq [2] \times [2]$.
        The top depicts $\Gamma_1(I)$, while the bottom depicts $\Gamma_2(I)$.
        Green edges represent joining paths, while all other edges represent cycle subpaths.}
    \end{figure}

    Next, we claim that $\Gamma_1(I)$ and $\Gamma_2(I)$ are $2$-regular.
    Indeed, note that $\Gamma_1(I)$ consists of three types of paths: joining paths, subpaths of $X$ connecting joining path endpoints, and subpaths of $Y$ connecting successive $Y$-fragments.
    By definition, joining paths are pairwise vertex-disjoint, and their internal vertices avoid $X \cup Y$.
    Moreover, subpaths of $X$ and $Y$ can only intersect at a vertex of $M$, but in $\Gamma_1(I)$, every subpath of $X$ lies inside an $X$-fragment and hence avoids $M$, so all subpaths of $X$ and $Y$ contained in $\Gamma_1(I)$ are pairwise vertex-disjoint.
    Thus, in $\Gamma_1(I)$, every internal vertex of a joining path has degree $2$, every internal vertex of a subpath of $X$ or $Y$ has degree $2$, and every endpoint of a joining path is incident to one edge of a joining path and one edge of $X$ or $Y$ and hence has degree $2$.
    A symmetric argument yields the same conclusion for $\Gamma_2(I)$.

    Hence, $\Gamma_1(I)$ and $\Gamma_2(I)$ are both disjoint unions of cycles.
    We then claim the following, which will be verified in Appendix \ref{section: computation}.
    \begin{claim}
    \label{claim: route}
        There is an even routing $I$ for which $\Gamma_1(I)$ and $\Gamma_2(I)$ are connected.
    \end{claim}
    Assuming this claim, choose $I$ such that $\Gamma_1(I)$ and $\Gamma_2(I)$ are both simple cycles in $G$.
    By construction, an edge of $X$ is either in $\Gamma_1(I)$ or $\Gamma_2(I)$; the same is true in $Y$.
    Therefore, as $I$ is nonempty, $\sum_{(i,j) \in I}(|E(P_{ij}^{\mathcal B})|+|E(P_{ij}^{\mathcal A})|)>0$ and
    \[ |E(\Gamma_1(I))|+|E(\Gamma_2(I))|=|E(X)|+|E(Y)|+2\sum_{(i,j) \in I}(|E(P_{ij}^{\mathcal B})|+|E(P_{ij}^{\mathcal A})|)>2\ell, \]
    so one of the cycles $\Gamma_1(I)$ or $\Gamma_2(I)$ is longer than $\ell$, a contradiction.
\end{proof}
\subsection{Proof of Theorem \ref{thm: main}}
Since $\mathcal F$ contains no $K_{4,T}'$, Theorem \ref{thm: exsubdiv} yields $|E(\mathcal F)|\leq C|V(\mathcal F)|^{11/8}$, where $C$ is a constant depending only on $T$.
Applying Lemma \ref{lem: auxgraph}, we have that $|V(\mathcal F)|\leq 2m$ and $|E(\mathcal F)|\geq k-m$, so
\[ k-m\leq |E(\mathcal F)|\leq C(2m)^{11/8}. \]
Therefore, $k=O(m^{11/8})$ and hence $m=\Omega(k^{8/11})$. \qed
\printbibliography

@article{chen1998smith,
    author = {Guantao Chen and Ralph J. Faudree and Ronald J. Gould},
    title = {Intersections of Longest Cycles in $k$-Connected Graphs},
    journal = {Journal of Combinatorial Theory, Series B},
    year = {1998},
    volume = {72},
    number = {1},
    pages = {143-149}
}

@article{groenland2025smith,
    author = {Carla Groenland and Sean Longbrake and Raphael Steiner and J\'er\'emie Turcotte and Liana Yepremyan},
    title = {Longest Cycles in Vertex-Transitive and Highly Connected Graphs},
    journal = {Bulletin of the London Mathematical Society},
    year = {2025},
    volume = {57},
    number = {10},
    pages = {2975-2990}
}

@article{ma2025smith,
    author = {Jie Ma and Ziyuan Zhao},
    title = {Intersections of Longest Cycles in Vertex-Transitive and Highly Connected graphs},
    journal = {arXiv preprint arXiv:2508.17438},
    year = {2025}
}

@incollection{bondy1995graphs,
    author = {J. A. Bondy},
    title = {Basic Graph Theory: Paths and Circuits},
    booktitle = {Handbook of Combinatorics (Vol. 1)},
    isbn = {0262071703},
    publisher = {MIT Press},
    year = {1996},
    pages = {3-110}
}

@article{stewart1995smith,
    author = {Iain A. Stewart and Ben Thompson},
    title = {On the Intersections of Longest Cycles in a Graph},
    journal = {Experimental Mathematics},
    year = {1995},
    volume = {4},
    number = {1},
    pages = {41-48}
}

@article{gutierrez2024smith,
    author = {Juan Guti\'errez and Christian Valqui},
    title = {On Two Conjectures About the Intersection of Longest Paths and Cycles},
    journal = {Discrete Mathematics},
    year = {2024},
    volume = {347}, 
    number = {11}
}

@article{grotschel1984smith,
    author = {Martin Gr\"otschel},
    title = {On Intersections of Longest Cycles},
    journal = {Graph Theory and Combinatorics: Proceedings of the Cambridge Combinatorial Conference in Honour of Paul Erdős},
    editor = {B\'ela Bollob\'as},
    year = {1984},
    pages = {171-189}
}

@article{conlon2021subdiv,
    author = {David Conlon and Oliver Janzer and Joonkyung Lee},
    title = {More on the Extremal Number of Subdivisions},
    journal = {Combinatorica},
    year = {2021},
    volume = {41},
    pages = {465-494}
}

@article{zamfirescu2001survey,
    author = {Tudor Zamfirescu},
    title = {Intersecting Longest Paths or Cycles: A Short Survey},
    journal = {Analele Universității Din Craiova: Seria Matematică-Informatică},
    year = {2001},
    volume = {28},
    pages = {1-9},
    url = {http://tzamfirescu.tricube.de/TZamfirescu-151.pdf}
}

@article{shabbir2013survey,
    author = {Ayesha Shabbir and Carol T. Zamfirescu and Tudor I. Zamfirescu},
    title = {Intersecting Longest Paths and Longest Cycles: A Survey},
    journal = {Electronic Journal of Graph Theory and Applications},
    year = {2013},
    volume = {1},
    number = {1},
    pages = {56-76}
}

@article{hou2026subdiv,
    author = {Jianfeng Hou and Yindong Jin and Donglei Yang and Fan Yang},
    title = {Balanced Subdivisions and Cycle Lengths in $K_{s,t}$-Free Graphs},
    journal = {European Journal of Combinatorics},
    year = {2026},
    volume = {137}
}
\appendix
\section{Discussion of Computer Search for Claim \ref{claim: route}}
\label{section: computation}
We now discuss the implementation of our computer search method, which is attached as the ancillary file \texttt{routing\_computation.ipynb}.

\subsection{Cell 1}
Here, we set up the search process. 

By Definition \ref{def: homogeneous}, each homogeneous $K_{4,8}'$-arrangement can be encoded via $(\pi,\varepsilon,\sigma,\eta) \in \mathfrak S_4 \times \{ \pm1 \}^4 \times \mathfrak S_8 \times \{ \pm1 \}^4$.
Here, $\pi \in \mathfrak S_4$ and $\sigma \in \mathfrak S_8$ are as in conditions \ref{condition (i)} and \ref{condition (iii)}, respectively, while $\varepsilon \in \{ \pm1 \}^4$ is the unique signing for which
\[ \varepsilon(i)=\begin{cases}
    +1 & \text{ if $b_{i1}\prec_X b_{i2}\prec_X \cdots \prec_X b_{i8}$} \\
    -1 & \text{ if $b_{i8}\prec_X b_{i7}\prec_X \cdots \prec_X b_{i1}$},
\end{cases} \]
and $\eta \in \{ \pm 1 \}^4$ is the unique signing for which
\[ \eta(i)=\begin{cases}
    +1 & \text{ if $y_{ij}^+\prec_Y y_{ij}^-$ for all $j \in [8]$} \\
    -1 & \text{ if $y_{ij}^-\prec_Y y_{ij}^+$ for all $j \in [8]$}.
\end{cases} \]
We refer to $(\pi,\varepsilon)$ as an \emph{$X$-type} and $(\sigma,\eta)$ as a \emph{$Y$-type} since each pair determines the relevant order of endpoints in its associated cycle.

First, we enumerate all valid $\sigma \in \mathfrak S_8$, of which there are only $1296$, as certain permutations prescribe contradictory orderings of all $32$ subdividing fragments with $\prec_Y$.
The function \texttt{build\_order} checks for contradictions in this global ordering prescribed by a given $\sigma \in \mathfrak S_8$.
In particular, it repeatedly chooses a subdividing fragment that has no predecessor and appends it to the global ordering.
Either the process terminates by successfully ordering all $32$ subdividing fragments and returns them in increasing order, or at some point, the process halts because every remaining subdividing fragment has a predecessor (and hence the remaining comparisons contain a directed cycle) and rejects $\sigma$.

We also record the following combinatorial fact, which is tested in \texttt{column\_orders\_agree}.
For each $p<q$, consider $(Y_{1p},Y_{2p},Y_{3p},Y_{4p},Y_{1q},Y_{2q},Y_{3q},Y_{4q})$.
By construction, $\sigma \in \mathfrak S_8$ corresponds to an ordering of the entries of this list with respect to $\prec_Y$.
Restricting this order to the first $4$ entries gives a permutation $\alpha \in \mathfrak S_4$ such that $ Y_{\alpha(1),p}\prec_Y \cdots \prec_Y Y_{\alpha(4),p}$.  
We let $\beta \in \mathfrak S_4$ denote the analogous restriction to the last $4$ entries.
Since $\sigma$ is constant across all $p<q$, so are $\alpha$ and $\beta$.
We claim that $\alpha=\beta$.
Indeed, let $p<q<r$ be indices.
Considering the pair $(p,q)$, we deduce $Y_{\beta(1),q}\prec_Y \cdots \prec_Y Y_{\beta(4),q}$, while considering the pair $(q,r)$, we deduce $Y_{\alpha(1),q}\prec_Y \cdots \prec_Y Y_{\alpha(4),q}$, which forces $\alpha=\beta$, as claimed.

Next, we enumerate even routings using bitmasks, i.e. each set $I \subseteq [4] \times [8]$ is stored as a $32$ bit integer, with one bit per $(i,j) \in I$.
We also normalize these routings by assuming that the second coordinates occurring in $I$ is of the form $[h]$ for some $h$.
It turns out that one can always find a suitable even routing even with this restriction; in fact, one can show that if the second coordinates are $j_1<j_2<\cdots<j_h$, then replacing each $j_i$ with $i$ yields another suitable even routing.
Both aspects are implemented in \texttt{make\_routings\_by\_size}, which also organizes the routings $I$ by their size $|I|$.

Furthermore, to avoid encoding unnecessary joining path data, we contract each path of $P(I)=\{ P_{ij}^{\mathcal B}: (i,j) \in I \} \cup \{ P_{ij}^{\mathcal A}:(i,j) \in I \}$ into a single pseudovertex. 
Letting $|I|=s$, this gives $2s$ pseudovertices.
Since edge contraction preserves connectedness of a graph, connectedness of the contracted $\Gamma_i(I)$ is equivalent to connectedness of the original $\Gamma_i(I)$.
We then note that after contracting the joining paths, the subpaths of $X$ in $\Gamma_i(I)$ form a perfect matching of the pseudovertices, while the subpaths of $Y$ in $\Gamma_i(I)$ form another perfect matching of the pseudovertices.
Thus, to check connectedness of $\Gamma_i(I)$, we alternately follow edges from the matchings and test if we obtain a single cycle on all $2s$ pseudovertices.
This strategy is implemented in \texttt{matching\_union\_connected}.

Since the number of $Y$-types is much larger than the number of $X$-types, we break up the computation as follows.
The function \texttt{build\_x\_data} constructs the matching with subpaths of $X$ in $\Gamma_1(I)$, the analogous matching in $\Gamma_2(I)$, and the matching with subpaths of $Y$ in $\Gamma_2(I)$; in addition, the function checks if $\Gamma_2(I)$ is connected.
Separately, \texttt{gamma1\_connected} checks if $\Gamma_1(I)$ is connected using an implicit construction of the matching with subpaths of $Y$ in $\Gamma_1(I)$ and leveraging the other matching in $\Gamma_1(I)$ already constructed by \texttt{build\_x\_data}. 
\subsection{Cell 2}
We search for a suitable even routing for every homogeneous $K_{4,8}'$-arrangement. 
We split this search into two parts since even routings of smaller size deal with almost all arrangements.

The idea behind \texttt{search\_small\_routes} is as follows.
Fixing an $X$-type, we consider normalized even routings $I$ of sizes $4,6,8,10$, and we discard $I$ unless $\Gamma_2(I)$ is connected.
For each surviving routing $I$, we compute the set of $Y$-types for which $\Gamma_1(I)$ is connected; let $C(I)$ denote this set.
Let $U$ be the union of $C(I)$ across all surviving routings $I$.
We then prune out some surviving routings using \texttt{greedy\_subcover} so that the union of $C(I)$ across all $I$ in this reduced list of routings is still $U$.
This process is repeated independently for each $X$-type.

We resume from where we left off in \texttt{search\_large\_routes}.
In particular, it takes an abridged approach by considering normalized even routings of sizes $12,14,16$ only to resolve the arrangements not dealt with by \texttt{search\_small\_routes}.

\subsection{Cell 3}
The last cell verifies that Claim \ref{claim: route} holds.
In particular, for each $X$-type, we retain the list of routings that we found previously.
The function \texttt{verify\_routes} then checks that every routing in this list is a nonempty even routing and makes $\Gamma_2(I)$ connected for that $X$-type.
It then verifies, for every $Y$-type, that some routing in the same list also makes $\Gamma_1(I)$ connected.
Consequently, for each homogeneous $K_{4,8}'$-arrangement, there is an even routing $I$ that allows both $\Gamma_1(I)$ and $\Gamma_2(I)$ to be connected.

\subsection{Results}
Running all cells in the notebook gives an output of \texttt{PASS} from the last cell, meaning that the computer search has successfully proven Claim \ref{claim: route}.
The computation should take a few minutes on a standard modern computer.
\end{document}